\documentclass[11pt]{article}
\usepackage{amssymb}
\usepackage{amsthm}
\usepackage{latexsym,array}
\usepackage{amsthm}
\usepackage{amsmath, mathrsfs}
 \newtheorem{theorem}{Theorem}[section]

 \newtheorem{proposition}[theorem]{Proposition}
 \newtheorem{definition}[theorem]{Definition}
 \newtheorem{remark}[theorem]{Remark}
 
\usepackage{color}
\begin{document}
\title{\bf On the Gauss-Lucas theorem in the quaternionic setting}

\author{Sorin G. Gal\\
University of Oradea\\
Department of Mathematics\\ and Computer Science\\
Str. Universitatii Nr. 1\\
410087 Oradea,
Romania \\
galsorin23@gmail.com
\and
J. Oscar Gonz\'alez-Cervantes\\
Departamento de Matem\'aticas  \\
 E.S.F.M. del
I.P.N. 07338 \\
M\'exico D.F., M\'exico\\
jogc200678@gmail.com
\and Irene Sabadini\\
Dipartimento di Matematica\\
Politecnico di Milano\\
Via Bonardi 9\\
20133 Milano, Italy\\
irene.sabadini@polimi.it}
\maketitle

\begin{abstract}

In theory of one complex variable, Gauss-Lucas Theorem states that the critical points of a non constant  polynomial  belong to the convex hull of the set of  zeros of the polynomial. The exact analogue of this result cannot hold, in general, in the quaternionic case; instead, the critical points  of a non constant  polynomial  belong to the convex hull of the set of  zeros of the so-called symmetrization of the given polynomial. An  incomplete proof of this statement was
given in \cite{V}. In this paper we present a different but complete proof of this theorem  and we discuss a consequence.
 \end{abstract}



\section{Introduction}
The  Gauss-Lucas Theorem, e.g. see \cite{A}, \cite{E}, \cite{M},  states that the critical points of a non constant  complex polynomial $P$ i.e. the zeros of the derivative $P'$, belong to the convex hull of the set of  zeros of $P$. It is a useful tool in classical complex analysis with several applications, for example in approximation theory.   A well known  direct consequence of the Gauss-Lucas theorem  is that a complex polynomial  has a zero of modulus which is at least a constant related with the coefficients of the polynomial.
An equivalent form of the Gauss-Lucas Theorem states that given any  non constant complex polynomial $P$, then any zero of $P'$ is a convex combination of elements of $Z_P$, the set of zeros of $P$. In other words, if $P'(z)=0$, then there exist $z_1,\dots, z_n \in Z_p$ and there exist $\lambda_1\dots , \lambda_n \in [0,1]$ such that  {}  $z=\lambda_1 z_1 +\cdots +\lambda_n z_n$, {} such that  {} $\lambda_1+\cdots+ \lambda_n=1$.

The theory of quaternionic polynomials with coefficients on one side, for example on the right, is rather rich (see \cite{lam} and the references therein) and it had a resurgence of interest in recent times, after that it has been observed that one sided quaternionic polynomials are a special case of slice regular functions, see \cite{CSS,GSS}. Quaternionic polynomials have some peculiarities, for example they can admit an infinite number of zeros. To provide an example, it is sufficient to consider the polynomial $P(q)=q^2+1$ whose roots are the purely imaginary quaternions of norm $1$. In general, one can show that the zeros of a quaternionic polynomial are either points or 2-dimensional spheres in the set of quaternions identified with $\mathbb R^4$. In this case, we say that the polynomial has a spherical zero. The number of zeros is related with the degree of a polynomial as it is equal to the number of isolated zeros plus twice the number of spheres.

The behaviour of the derivative of a quaternionic polynomial may be unexpected. For example, $P(q)=q^2-q(\mathbf{i}+\mathbf{j})+\mathbf{k}$ has a zero of multiplicity $2$ at the point $\mathbf{i}$ but its derivative $P'(q)=2q- (\mathbf{i}+\mathbf{j})$ does not vanish at this point.

It is then natural to ask whether the Gauss-Lucas theorem holds directly in this framework.
In \cite{V} the author states that the critical points of a the derivative of a polynomial belong to the convex hull of the symmetrization of the polynomial. However the proof, which is done in a few lines, is very incomplete. In this paper we give a complete proof of this theorem and one direct consequence of this result.

\section{Preliminary results on quaternionic polynomials}

By $\mathbb H$ we denote the algebra of real quaternions, namely the set of elements of the form
 $q=q_0+q_1 {\bf i }+ q_2 {\bf j } +q_3 {\bf k}$, where $q_n \in \mathbb R$ for $n=0,1,2,3$ and the imaginary units ${\bf  i},  {\bf j} , {\bf k}$ satisfy:  ${\bf i}^2={\bf  j}^2={\bf k}^2=-1, \  {\bf ij}=-{\bf ji}={\bf k}, \  {\bf jk}=-{\bf kj}={\bf i}, \ {\bf ki}=-{\bf ik} ={\bf j}$. Moreover, $q_0$ is called the real part of $q$ and it is often denoted by ${\rm Re}(q)$ while ${\bf q}=q_1 {\bf i} + q_2{\bf j }+q_3 {\bf k}$ is called the vector part, or imaginary part. The norm of a quaternion is defined as $|q|=(q_0^2+q_1^2+q_2^2+q_3^2)^{1/2}$ and corresponds to the Euclidean norm when we identify $q$ with the element $(q_0,q_1,q_2,q_3)\in\mathbb R^4$.  The  purely imaginary quaternions ${\bf q}$ with norm $1$ form a $2$-dimensional sphere $\mathbb S^2$. Any $I\in\mathbb S^2$ is such that $I^2=-1$ and thus $\mathbb C(I)=\{x+Iy\ |\ x,y\in\mathbb R\}$ is a complex plane whose imaginary unit is $I$. Given any nonreal quaternion $q$ with nonzero vector part belongs to the complex plane $\mathbb C (I_q)$ where $I_q=q_1 {\bf i} + q_2{\bf j }+q_3 {\bf k}/|q_1 {\bf i} + q_2{\bf j }+q_3 {\bf k}|$. If a quaternion $q=x\in\mathbb R$, then $q=x+I\cdot 0$ for any $I\in\mathbb S^2$.
\\
Given a quaternion $\alpha$ written in the form $\alpha=a_0+Ia_1$, $a_0,a_1\in\mathbb R$, the set of elements of the form $a_0+Ja_1$ when $J\in\mathbb S^2$ forms a $2$-dimensional sphere denoted by $[\alpha]$. This $2$-sphere can equivalently be described as the set of all the elements of the form $r^{-1}\alpha r$ where $r\in\mathbb H\setminus \{0\}$.

In this paper, we will consider quaternionic polynomials with coefficients written on the right, i.e. $$\displaystyle P(q)= \sum_{n=0}^m q^n a_n, \qquad q\in\mathbb H,$$
where $a_n\in\mathbb H$ for $n=0,\dots,m$. These polynomials can be multiplied via the so-called $*$-multiplication by taking the convolution of their coefficients: if {$\displaystyle Q(q)=\sum_{n=0}^r q^n b_n$} then $$(P*Q)(q)=\sum_{n=0}^{m+r} q^n (\sum_{s+k=n} a_sb_k).$$ We note that this multiplication coincides with the pointwise multiplication only when the coefficients of $P$ are real, otherwise we have $(P*Q)(q)=0$ if $P(q)=0$ and if $P(q)\not=0$
$$
(P*Q)(q)=P(q)Q(P(q)^{-1}qP(q),
$$
see \cite{CSS} and \cite{GSS}.\\
Using the $*$-multiplication we can define the polynomial $P^s$ which will be crucial in the sequel:
\begin{definition} \label{def1} Given a quaternionic polynomial $P(q)= \sum_{n=0}^m q^n a_n$ we define the following polynomials:
$$\displaystyle P^c(q)= \sum_{n=0}^m q^n \overline{a}_n;$$
$$ \displaystyle P^s(q)= P(q)*P^c(q) = P(q)^c*P(q) = \sum_{n=0}^{2m} q^n \sum_{s+r=n} a_s \overline{a}_{r}.$$

\end{definition}
 The following result, see \cite{gm}, \cite{lam}, summarizes some important properties of the zeros of a quaternionic polynomial:
     \begin{theorem}
\begin{enumerate}
\item
A quaternion $\alpha$ is a zero of a (nonzero) polynomial $P(q)$ if and only if
the polynomial $q - \alpha$ is a left divisor of $P(q)$.
\item A sphere $\alpha$ consists of zeros of a (nonzero) polynomial $P(q)$ if and only if
the polynomial $q - 2{\rm Re}(\alpha) q +|\alpha|^2$ is a divisor of $P(q)$.
\item If $P(q) = (q - \alpha_1)*\ldots * (q - \alpha_n)$, where $\alpha_1, \ldots , \alpha_n \in\mathbb H$,
then $\alpha_1$ is a zero of $P$ and every other zero of $P$ belong to the spheres  $[\alpha_i]$, $i=2,\ldots, n$.
\item
If $P$ has two distinct zeros in an equivalence class $[\alpha]$, then all the elements in $[\alpha]$ are
zeros of $P$. In particular, a polynomial with real coefficients has either real or spherical zeros.
\end{enumerate}
\end{theorem}
We also note the following:
\begin{remark}\label{rem11}{\rm The zeros of the polynomial $P^s$ are the spheres (which may be reduced to real points) associated to the zeros of $P$, see \cite{CSS}, \cite{GSS}. In other words, the sphere $[\alpha]=[a_0+Ia_1]$, $a_1\not=0$,  consists of zeros of $P^s$ if and only if there exists $J\in\mathbb{S}^2$ such that $a_0+Ja_1$ is a zero of $P$. The real point $\alpha$ is a zero of $P^s$ if and only if it is a zero of $P$.
}
\end{remark}
\section{A quaternionic version of the Gauss-Lucas Theorem}

The derivative of $P(q)$ is defined in the customary way as $P'(q)=\sum_{n=1}^m q^{n-1}na_n$.
By $Z_P$ we denote the set of zeros of $P$ and the zeros of $P'$ are called critical points of $P$.\\
Let  $U\subset\mathbb H$  be a non empty set, then     $Kull(U)$  is the    intersection of all convex subsets of $\mathbb H$ containing $U$, and it  is called the convex hull of  $U$.
		Since  $\mathbb H$ can be naturally identified with $\mathbb R^4$,  it follows   that  $q\in Kull(U)$  if and only if $q$ is a convex combination of elements of $U$.

The following result is not surprising since it deals with polynomials with real coefficients, and thus it mimics the result in the classical case:
\begin{proposition}\label{real}
Given  $$\displaystyle P(q)= \sum_{n=0}^m q^n a_n, $$ where $a_n \in \mathbb R$ for all $n=1,\dots, m$. Then the critical points of $P$ belong to  $Kull(Z_P)$.
\end{proposition}

 \begin{proof} It relies on the classical Gauss-Lucas theorem in the complex case. Consider $v\in \mathbb H$ such that $\displaystyle 0=P'(v)=  \sum_{n=1}^m v^{n-1}n a_n$. For a suitable $I\in\mathbb S^2$, we can write $v=x+Iy$ where $x,y\in \mathbb R$. Then $v\in\mathbb C(I)$ is a critical point of $P\mid_{\mathbb C(I)}$, which is a complex polynomial. From the classical Gauss-Lucas Theorem one has that $v$ is a convex combination of elements of $Z_{P\mid_{\mathbb C(I)}}$ and thus, in particular, $v\in Kull(Z_P)$ and the assertion follows.
 \end{proof}
\begin{remark}{\em
The previous result fails when the coefficients of a polynomial are not all real numbers. Let, for example,
$$ P(q)= q^2 \frac{1}{2} + q{\bf i }+  {\bf  j}=\frac 12(q+{\bf i}+{\bf ij})*(q+{\bf i}-{\bf ij}).$$
The equation $(q-a)*(q-b)$ has zeros $q=a$ and $q=(b-\bar a)^{-1}b(b-\bar a)$ so, in this case there is just the zero $-{\bf i}-{ij}$ with multiplicity two. It follows that $Z_P=\{  -  {\bf  i} - {\bf  ij}\}$ and  $Z_{P'}=\{- {\bf  i } \} $. Therefore,  $Z_{P'}\not\subset Kull(Z_P)= \{- {\bf  i } - {\bf  ij} \}$.
}
\end{remark}
In order to prove the appropriate generalization of the Gauss-Lucas theorem for quaternionic polynomial, namely that  $Z_{P'}\subset Kull(Z_{P^s})$, we need two technical results:
\begin{proposition}\label{prop1}
Let $P(q)$ be a quaternionic polynomial.
Then $Z_{P'\mid_{\mathbb C(I)}} \subset Kull(Z_{P^s\mid_{\mathbb C(I)}})$, for each $I\in\mathbb S^2$ if and only if $Z_{P'} \subset Kull(Z_{P^s})$.
\end{proposition}

 \begin{proof}
For of all, let us note that if $Z_{P'\mid_{\mathbb C(I)}} \subset Kull(Z_{P^s\mid_{\mathbb C(I)}})$, for every $I\in\mathbb S^2$ then $Z_{P'\mid_{\mathbb C(I)}} \subset Kull(Z_{P^s\mid_{\mathbb C(I)}})\subseteq Kull(Z_{P^s})$. Thus
$$Z_{P'}=\bigcup_{I\in\mathbb S^2}Z_{P'\mid_{\mathbb C(I)}} \subset  Kull(Z_{P^s}).$$

To prove the converse, we assume that  $Z_{P'} \subset  Kull(Z_{P^s})$ and we consider $q\in Z_{P'\mid_{\mathbb C(I)}}$ and we write $q=x+Iy$. Since $q$ is a zero of $P'$ we have $q \in Kull(Z_{P^s})$ and we can write $$q=\sum_{l=1}^n \rho_l a_l + \sum_{k=1}^m \delta_k b_k ,$$
where  $a_1,\dots, a_n \in  Z_{P^s}  \cap \mathbb C(I)$, {} $b_1,\dots, b_m \in  Z_{P^s}  \setminus \mathbb C(I)$, {} $0\leq \rho_1, \dots \rho_n,\delta_1, \dots \delta_m\leq 1 $, {} and   $\displaystyle 1=\sum_{l=1}^n \rho_l + \sum_{k=1}^m \delta_k $. Since $P^s$ has real coefficients, it has spherical zeros therefore  if $a_l= r_l+ I s_l $ and  $b_k=x_k+ I_k y_k$ with $r_l,s_l,x_k,y_k\in \mathbb  R$ and $I_k\in \mathbb S^2\setminus\{I\}$, then   $c_l = r_l+ I |s_l | $ and $d_k=x_k+ I|y_k| $ belong to $Z_{P^s} \cap \mathbb C(I) $, for $l=1,\dots, l$ and  $k=1,\dots, m$.
Let us set
$$
q'=\sum_{l=1}^n \rho_l c_l + \sum_{k=1}^m \delta_k d_k =x'+Iy'.
$$

 It is clear that $q', \overline{q'}\in Kull(Z_{P^s\mid_{\mathbb C(I)}}) $ and that $Re(q)=Re(q')$ while the imaginary part can be written as
  $$ y I=\sum_{l=1}^n \rho_l s_l  I  +     \sum_{k=1}^m \delta_k   I_k y_k . $$
Thus
 $$|y |\leq    |\sum_{l=1}^n \rho_l s_l |  +     |\sum_{k=1}^m \delta_k   I_k y_k |  \leq   \sum_{l=1}^n \rho_l |s_l  |  +     \sum_{k=1}^m \delta_k  | y_k |$$
$$\leq  \sum_{l=1}^n \rho_l | s_l |    +     \sum_{k=1}^m \delta_k  | y_k | = y' $$
Then since $-y' \leq y \leq y' $ i.e.  $q=x+Iy$ is in the segment with end points $q'=x'+Iy'$ and $ \overline{q'}=x'-Iy'$. Therefore   $q\in  Kull(Z_{P^s\mid_{\mathbb C(I)}})$.
\end{proof}

\begin{remark}{\rm
Given the quaternionic polynomial $P$, its restriction to a complex plane $\mathbb C(I)$ can be written as  $P\mid_{\mathbb C(I)}(z) = P_1(z) +P_2(z) J$, where $P_1$ and $P_2$ are complex-valued polynomials in $z\in\mathbb C(I)$ and $J\in\mathbb S^2$ orthogonal to $I$. Since, by linearity, $P'\mid_{\mathbb C(I)} = P'_1 +P'_2 J$ one has that   $Z_{P_1'}\cap Z_{P_2'}=  Z_{P'\mid_{\mathbb C(I)}}$.
 \\
 The previous Proposition \ref{prop1} gives that  $Z_{P'} \subset Kull(Z_{P^s})$ if and only if  $Z_{P'\mid_{\mathbb C(I)}} \subset Kull(Z_{P^s\mid_{\mathbb C(I)}})$, for every $I\in\mathbb S^2$, or equivalently  $Z_{P_1'}\cap Z_{P_2'}\subset Kull (Z_{P^s}\mid_{\mathbb C(I)} )$, for every $I\in\mathbb S^2$.
}
\end{remark}
Another useful formula which will be used below is the following, see e.g. \cite{CSS}:
\begin{equation}\label{Ps}
P ^s\mid_{\mathbb C(I)} (z) =P_1(z) \overline{P_1(\bar z)} +  P_2(z)  \overline{P_2(\bar z) },  \qquad \forall z\in\mathbb C(I).
\end{equation}

\begin{proposition}\label{pro2} Let $P_1$ and $P_2$ two complex polynomials and set $Q(z)=P_1(z)\overline{P_1(\bar z)} + P_2(z)\overline{P_2(\bar z)}$, for each $z\in\mathbb C$.  Then  $Z_{P_1'}\cap Z_{P_2'}\subset Kull(Z_Q)$.
\end{proposition}
Proof.  Set $\displaystyle P_1(z)= \sum_{n=0}^m z^n a_n $ and $\displaystyle P_2(z)= \sum_{n=0}^l z^n b_n $ and assume, without loss of generality, that $m\geq l$. Set $b_k=0$ for $k=l+1,\dots, m$ so that
\begin{equation}\label{equa21}  Q(z)= \sum_{n=0}^{2m} z^n \left[ \sum_{k=0}^n a_k \bar a_{n-k} + b_k \bar b_{n-k}  \right].\end{equation}
 Let us define $L(z)= P_1'(z)\overline{P_1(\bar z)} +  P_2'(z)\overline{P_2(\bar z)} $ for each $z\in\mathbb C$. Then $Z_{P_1'}\cap Z_{P_2'} \subset Z_{L}$ and  $Q'(z)=L(z)+\overline{L(\bar z)}$ for $z\in\mathbb C$. Moreover,
\[
zL(z)= z\left( \sum_{n=1}^m z^{n-1} na_n \right)  \left(  \sum_{n=0}^m z^n \bar a_n \right) +  z \left(  \sum_{n=1}^l z^{n-1}n b_n \right)\left(  \sum_{n=0}^l z^n \bar  b_n\right)
\]
\[=   \left( \sum_{n=0}^m z^{n} na_n \right)  \left(  \sum_{n=0}^m z^n \bar  a_n \right) +   \left(  \sum_{n=0}^l z^{n}n b_n \right)\left(  \sum_{n=0}^l z^n \bar  b_n\right)
\]
\[=   \sum_{n=0}^{2m} z^{n} \sum_{k=0}^n ka_k \bar a_{n-k} +  \sum_{n=0}^{2l} z^{n} \sum_{k=0}^n kb_k \bar b_{n-k},
\]
so that
\begin{equation}\label{equa22} zL(z) =   \sum_{n=0}^{2m} z^{n} \sum_{k=0}^n k ( \ a_k \bar a_{n-k} b_k \bar b_{n-k} \ ). \end{equation}
We now prove that there exists $c_0, c_1, \dots, c_m \in\mathbb C$ such that $$ \sum_{k=0}^n a_k \bar a_{n-k} + b_k \bar b_{n-k}  =\sum_{k=0}^n c_k \bar c_{n-k} , \quad k\in\{0,\dots, n\},$$
for each $n=0,1,\dots,m$. We have two cases:
\begin{enumerate}
\item If $a_0\neq 0$ or $b_0\neq 0$.
One can always choose  $c_0=x_0+iy_0\in\mathbb C$, with $x_0,y_0\in\mathbb R$ , such that $$ a_0 \bar a_0 + b_0 \bar b_0 = x_0^2+y_0^2= c_0 \bar c_0  .$$
Then we can continue and choose $c_1=x_1+iy_1 \in\mathbb C$, with $x_1,y_1\in\mathbb R$, such that $$ a_0 \bar a_1 + a_1\bar a_0 + b_0 \bar b_1 + b_1\bar b_0 = 2(x_0x_1+y_0+y_1) = c_0 \bar c_1 + c_1\bar c_0 .$$
Iterating the procedure, it is possible to find $c_n\in\mathbb C$ such that $$ \sum_{k=0}^n a_k \bar a_{n-k} + b_k \bar b_{n-k}= \sum_{k=0}^{n} c_k \bar c_{n-k}= c_0 \bar c_n + c_n\bar c_0 + \sum_{k=1}^{n-1} c_k \bar c_{n-k} $$
in fact it is possible to find $c_n$ such that
$$
 \sum_{k=0}^n a_k \bar a_{n-k} + b_k \bar b_{n-k}-\sum_{k=1}^{n-1} c_k \bar c_{n-k}=c_0 \bar c_n + c_n\bar c_0 .$$
Thus
$$ \sum_{k=0}^n a_k \bar a_{n-k} + b_k \bar b_{n-k}  =\sum_{k=0}^{n} c_k \bar c_{n-k} .$$
\item If  $a_0= 0$ and $b_0= 0$, let
$$r=\min\{ k\in \mathbb N \ \mid \  a_k\neq 0  {} \textrm{ or  } {} b_k\neq 0\}.$$
 In this case, we set  $c_0=\cdots=c_{r-1}=0$ and we look for $c_r\in\mathbb C$ such that
$$ \sum_{k=r}^{2r} a_k \bar a_{2r-k} + b_k \bar b_{2r-k}  = a_r\bar a_r + b_r\bar b_r =  c_r \bar c_r. $$
The following terms $c_k$, with $k>r$ are obtained reasoning as in (i).
\end{enumerate}

 Let us set $\displaystyle M(z)= \sum_{n=0}^{m} z^n c_n$, for $z\in\mathbb C$. Then formula (\ref{equa21}) gives \begin{equation}\label{equa24}
 Q(z)= M(z)\overline{M(\bar z)}, \quad \forall z\in \mathbb C.\end{equation}
 Therefore $Z_M \cup \overline{Z_M}= Z_{P^s} $, where $\overline{Z_M} =\{\bar z \mid \ z\in Z_M\}$. On the other hand, from  (\ref{equa22}) one obtains
$$zL(z) = \sum_{n=0}^{2m} z^{n} \sum_{k=0}^n k c_k \bar c_{n-k}  =  \left(  \sum_{n=0}^{m} z^{n}n c_n \right)\left( \sum_{n=0}^m z^n \bar c_n \right)= z \left(  \sum_{n=0}^{m} z^{n-1}n c_n \right)\left( \sum_{n=0}^m z^n \bar c_n \right),$$
so that
\begin{equation}\label{equa3} zL(z)=zM'(z)\overline{M(\bar z)},  \quad \forall z\in\mathbb C.\end{equation}
Thus given $w\in Z_{P_1'}\cap Z_{P_2'}$ we have two cases:
\begin{enumerate}
\item If $w\neq 0 $, recalling that $Z_{P_1}\cap Z_{P_2}\subset Z_{L}$, we have $w\in Z_{L}$. From (\ref{equa3}) it follows that either  $M'(w)=0 $ or $M(\bar w)=0$. Then $w\in Kull(Z_M)$ or $\bar w \in Z_M\subset Kull(Z_M)$. Using (\ref{equa24}) one obtains that  $w\in Kull (Z_{Q})$.

\item If $w=0$, then $Q'(w)=P_1'(w)\overline{P_1(\bar w)} +  P_1(w)\overline{P_1'(\bar w)}   +  P_2'(w)\overline{P_2(\bar w)}  + P_2(w)\overline{P_2'(\bar w)} =0  $ and, immediately, $w\in Kull(Z_Q)$.
\end{enumerate}
This concludes the proof. $\Box$
\\
We can now prove the main result:
\begin{theorem}
Let $P$ be a quaternionic polynomial. Then $Z_{P'}\subset Kull(Z_{P^s})$.
\end{theorem}

\begin{proof} Let us consider $P^s$, its restriction to the complex plane $\mathbb C(I)$ and the polynomial $Q(z)$ defined in Proposition \ref{pro2}. Then if we write $P|_{\mathbb C(I)}(z)=P_1(z)+P_2(z)J$ where $J\in\mathbb S^2$ and $J\perp I$ we have that $P^s|_{\mathbb C(I)}(z)=Q(z)$, see (\ref{Ps}). Proposition \ref{pro2} shows that $Z_{P'|_{\mathbb C(I)}}=Z_{P'_1}\cap Z_{P'_2}\subset Kull (P^s|_{\mathbb C(I)}    )$     and  the statement follows by Proposition \ref{prop1}.
\end{proof}

A well known   consequence of the complex Gauss-Lucas theorem, see \cite{E},  is that any complex polynomial $\displaystyle P(z) = \sum_{n=0}^{m} z^n b_n$, with $b_m\neq 0$,  has a zero of modulus at least
     \begin{equation}\label{GLin}
     \displaystyle   \max_{0< n \leq m-1} \left\{  \displaystyle   \left( \begin{array}{cc}m \\ n   \end{array}  \right)^{-1} \frac{  \displaystyle  |  b_{m-n}   |    }{  \displaystyle |b_{m} |   }     \right\}^{\frac{1}{n} } .
     \end{equation}
The generalization of this result in this framework is the following:

\begin{proposition}
  Given  $$\displaystyle P(q)= \sum_{n=0}^m q^n a_n, $$ where $a_n \in \mathbb H$ for all $n=1,\dots, m$,  such that $\displaystyle \sum_{k=0}^{2m} a_k \overline{a}_{2m-k}\neq 0$. Then $P$ or $P^c$  has a zero with  quaternionic modulus at least  	
	$$      \displaystyle   \max_{0< n \leq 2m-1} \left\{  \displaystyle   \left( \begin{array}{cc}2m \\ n   \end{array}  \right)^{-1} \frac{  \displaystyle  |   \sum_{k=0}^{2m-n} a_k \overline{a}_{2m-n -k}    |    }{  \displaystyle | \sum_{k=0}^{2m} a_k \overline{a}_{2m-k}   |   }     \right\}^{\frac{1}{\nu} } . $$
\end{proposition}

\begin{proof} By Definition \ref{def1} one has  that    $Z_{P} \cup Z_{P^c} =  Z_{P^s} $
 and  all coefficients of $  P^s$ are the real numbers $\displaystyle    b_n  = \sum_{k=0}^n a_k\bar{a}_{n-k}$ for $n=0,\dots, 2m$. Then $P^s$ restricted to each slice can be considered as a complex polynomial and it has a zero point of  modulus at least the constant given in (\ref{GLin}).
 \end{proof}

\end{document}